\documentclass[11pt,a4paper,reqno]{amsart}
\usepackage{newtxtext}

\usepackage{epsfig}
\usepackage{amsfonts}
\usepackage{amsmath}
\usepackage{amssymb}
\usepackage{amsthm}
\usepackage{times}
\usepackage{xcolor}
\usepackage{enumerate}
\usepackage[T1]{fontenc}
\usepackage[utf8]{inputenc}
\usepackage{graphicx}
\usepackage{nicefrac}
\usepackage[includehead,includefoot,margin=25mm]{geometry}
\usepackage[all]{xy}
\usepackage{multicol}
\usepackage[colorlinks=true,linkcolor=blue,citecolor=blue,urlcolor=blue]{hyperref}
\usepackage{cleveref}
\usepackage{comment}

\newtheorem{theorem}{Theorem}[section]
\newtheorem*{theorem*}{Theorem}

\newtheorem{lemma}[theorem]{Lemma}
\newtheorem{proposition}[theorem]{Proposition}

{\theoremstyle{remark}
  \newtheorem{remark}[theorem]{Remark}}
{\theoremstyle{definition}

  \newtheorem{example}[theorem]{Example}

}

\makeatletter \newcommand*\bigcdot{\mathpalette\bigcdot@{.5}} \newcommand*\bigcdot@[2]{\mathbin{\vcenter{\hbox{\scalebox{#2}{$\m@th#1\bullet$}}}}} \makeatother

\newcommand{\tam}{\operatorname{Tame}}

\newcommand{\aut}{\operatorname{Aut}}

\begin{document}

\title{On the Tame Isotropy Group of Locally Finite Derivations of $\mathbb{K}[X,Y]$}

\author{Luis Cid}
\address{Instituto de Matem\'atica y F\'\i sica, Universidad de Talca,
  Casilla 721, Talca, Chile.}%
\email{luis.cid@utalca.cl}

\author{Marcelo Veloso} 
\address{Departamento de Fisíca Estatística e Matemática, Universidade Federal de S\~ao Jo\~ao del-Rei, Brasil}
\email{veloso@ufsj.edu.br}

\date{\today}

\subjclass[2020]{13N15, 14R10, 13B10}

\keywords{Locally finite derivation, (Tame) Isotropy group, Exponential automorphism, Polynomial automorphism, Jordan decomposition}

\begin{abstract}
Let $\mathbb{K}$ be an algebraically closed field of characteristic zero.
We study the tame isotropy group $\tam_D(\mathbb{K}[X,Y])$ of locally finite derivations of the polynomial ring $\mathbb{K}[X,Y]$, using Van den Essen's classification up to conjugation. For each normal form, we explicitly determine the corresponding tame isotropy group. We then compare $\tam_D(\mathbb{K}[X,Y])$ with the tame isotropy group of the associated exponential automorphism $\exp(D)$, and prove that these groups always coincide. This stands in contrast to the behaviour of the full automorphism group, where such an equality may fail for derivations with a nontrivial semisimple part.
\end{abstract}

\maketitle

\section{Introduction}

We denote by $\mathbb{K}[X_1,\ldots,X_m]$ the polynomial ring in $m$ variables over an algebraically closed field $\mathbb{K}$ of characteristic zero. In some cases, we simply write $\mathbb{K}^{[m]}$. Let $B$ be an affine $\mathbb{K}$-algebra. We denote by $\operatorname{End}(B)$ the monoid of $\mathbb{K}$-algebra endomorphisms of $B$, and by $\operatorname{Aut}(B)$ the group of $\mathbb{K}$-automorphisms of $B$.

An \textit{elementary automorphism} of $\mathbb{K}^{[m]}$ is an automorphism of the form
\[
(X_1,\dots,X_m)\longmapsto(X_1,\dots,aX_i+f_i,\dots,X_m),
\]
where $i\in\{1,\dots,m\}$, $a\in\mathbb{K}\setminus\{0\}$, and $f_i\in\mathbb{K}[X_1,\dots,X_m]$ is independent of $X_i$. We denote by $\operatorname{EL}(\mathbb{K}^{[m]})$ the set of elementary automorphisms of $\mathbb{K}^{[m]}$. Recall that an automorphism of $\mathbb{K}^{[m]}$ is called \textit{tame} if it can be expressed as a composition of elementary automorphisms; otherwise, it is called \textit{wild}. A classical and highly nontrivial result (see \cite{nagata73,Re1968,van2000polynomial}) states that every automorphism of $\mathbb{K}[X_1,X_2]$ is tame. In contrast, in the polynomial ring $\mathbb{K}[X_1,X_2,X_3]$, wild automorphisms do exist (see \cite{shestakov042}).

A derivation on $B$ is a $\mathbb{K}$-linear map $D\colon B\to B$ satisfying the Leibniz rule
\[
D(ab)=aD(b)+bD(a), \qquad \text{for all } a,b\in B.
\]
We denote by $\operatorname{Der}(B)$ the set of all derivations of $B$. 

In \cite{freitas25}, the authors introduce the notion of the tame isotropy group of a derivation. They compute the tame isotropy group of triangular derivations over $\mathbb{K}^{[m]}$ for $m \leq 3$, and show that a Shamsuddin derivation $D$ is simple if and only if $\tam_D(\mathbb{K}^{[m]})$ is trivial.

Given $D\in\operatorname{Der}(\mathbb{K}^{[m]})$, the \textit{tame isotropy group of $D$} is the subgroup generated by all elementary automorphisms that commute with $D$, denoted by $\tam_D(\mathbb{K}^{[m]})$, that is,
\[
\tam_D(\mathbb{K}^{[m]})=\left\langle \varphi\in\operatorname{EL}(\mathbb{K}^{[m]}) \,\middle|\, \varphi D = D\varphi \right\rangle,
\]
where $\langle S\rangle$ denotes the subgroup generated by $S$.

The \emph{isotropy group} of $D$ is defined by
\[
\operatorname{Aut}_D(B)
=
\{\varphi\in \operatorname{Aut}(B)\mid \varphi D=D\varphi\}.
\]

It follows from the definitions that 
\[
\tam_D(\mathbb{K}^{[m]}) \subset \operatorname{Aut}_D(\mathbb{K}^{[m]}).
\]
In \cite{CV26}, the authors determine $\operatorname{Aut}_D(B)$ when $D$ is a locally finite derivation, as well as the isotropy group of the exponential automorphism associated with $D$.

The main purpose of this paper is to determine $\tam_D(\mathbb{K}[X,Y])$ for every locally finite derivation $D$ of $\mathbb{K}[X,Y]$, using Van den Essen's classification \cite{V92}. We also study the relationship between $\tam_D(\mathbb{K}[X,Y])$ and the tame isotropy group of the exponential automorphism $\exp(D)$. Recall that $\exp(D)$ can be viewed as the time-one map of the polynomial vector field associated with $D$; this connection, studied over $\mathbb{R}$ and $\mathbb{C}$ by Bass and Meisters \cite{BassMei85} in the context of polynomial flows, motivates the comparison between the isotropy of $D$ and that of $\exp(D)$.
 
In \cite{CV26}, it is shown that the equality
\[
\operatorname{Aut}_D(B)=\operatorname{Aut}_{\exp(D)}(B)
\]
may fail when $D$ has a nontrivial semisimple part. One of our main results shows that this phenomenon does not occur within the tame subgroup: for every locally finite derivation $D$ of $\mathbb{K}[X,Y]$, we have
\[
\tam_D(\mathbb{K}[X,Y])=\tam_{\exp(D)}(\mathbb{K}[X,Y]).
\]

The paper is organized as follows. Section~\ref{sec:generalities} recalls the necessary background on locally finite derivations, their Jordan decomposition, and the exponential map. Section~\ref{sec:tame} determines $\tam_D(\mathbb{K}[X,Y])$ for each normal form. Section~\ref{sec:tame-exp} compares the tame isotropy group of $D$ with that of $\exp(D)$ and establishes their equality in all cases.

%%%%%%%%%%%%%%%%%%%%%%%%%%%%%%%%%%%%%%%%%%%%%%%%%%%%%%%%%%%%%%%%%%%%%%%%%%%%%

\section{Generalities} 
\label{sec:generalities}

In this section, we recall the main concepts that will be addressed in the paper.

\subsection{Locally finite derivations and their classification}

An endomorphism $E \in \operatorname{End}(B)$ is said to be \emph{locally finite} if, for every $b \in B$, the $\mathbb{K}$-vector space $\operatorname{span}_{\mathbb{K}}\{E^n(b) \mid n \ge 0\}$ is finite-dimensional. Similarly, a derivation $D \in \operatorname{Der}(B)$ is called \emph{locally finite} if, for every $b \in B$, the set $\{D^n(b) \mid n \ge 0\}$ spans a finite-dimensional $\mathbb{K}$-subspace of $B$. We denote by $\operatorname{LFD}(B)$ the set of all locally finite derivations of $B$, and by $\operatorname{LFA}(B)$ the set of all locally finite automorphisms of $B$.

Two important classes of locally finite derivations are given by semisimple derivations those admitting a basis of eigenvectors and locally nilpotent derivations, for which every element is annihilated by a sufficiently high iterate. These classes represent, in a natural sense, opposite extremes within the class of locally finite derivations.

It is straightforward to verify that every locally finite derivation on the polynomial ring in one variable, $\mathbb{K}[X]$, is of the form
\[
D = (aX + b)\frac{\partial}{\partial X}, \quad \text{with } a,b \in \mathbb{K}.
\]
In this case, the determination of the tame isotropy group of such a derivation is immediate, as will be illustrated in the following example.

\begin{example}
Let $D$ be a nonzero locally finite derivation on  $\mathbb{K}[X]$. Then $D$ must be of the form $D = (aX + b)\dfrac{\partial}{\partial X}$. Let $\rho \in \operatorname{Aut}(\mathbb{K}[X])$, where any automorphism is given by $\rho(X) = \alpha X + \beta$ with $\alpha \neq 0$. Suppose $\rho$ commutes with $D$, i.e., $\rho D = D \rho$. Then we compute:
\[
\rho(D(X)) = \rho(aX + b) = a(\alpha X + \beta) + b = \alpha a X + \beta a + b,
\]
\[
D(\rho(X)) = D(\alpha X + \beta) = \alpha(aX + b) = \alpha a X + \alpha b.
\]
Equality of both expressions yields the condition $\beta a + b = \alpha b$. Therefore, 
\begin{itemize}
    \item $\operatorname{Aut}_D(\mathbb{K}[X])=\{X + \beta \mid \beta \in \mathbb{K}\}$, if $a = 0$ (case when $D$ is LND). Or
    \item $\operatorname{Aut}_D(\mathbb{K}[X])=\{\alpha X + \frac{(\alpha -1)b}{a}\mid \alpha \in \mathbb{K}^{\ast}, \mbox{ and } \alpha \neq 0$\}, if $a \neq 0$.
\end{itemize}
\end{example}

In dimension two, locally finite derivations are classified up to conjugation by the
following result.

\begin{lemma}[Corollary~4.7 of \cite{V92}]\label{lfdclas}
Let $D\neq 0$ be a locally finite derivation on $\mathbb{K}[X,Y]$. Then there exists
$\varphi\in\operatorname{Aut}(\mathbb{K}[X,Y])$ such that $\varphi D\varphi^{-1}$ is one of the
following:
\begin{enumerate}
  \item $D=f(X)\dfrac{\partial}{\partial Y}$, \quad $f(X)\in\mathbb{K}[X]$, $f\neq 0$;
  \item $D=\dfrac{\partial}{\partial X}+bY\dfrac{\partial}{\partial Y}$,
        \quad $b\in\mathbb{K}^*$;
  \item $D=aX\dfrac{\partial}{\partial X}+(amY+X^m)\dfrac{\partial}{\partial Y}$,
        \quad $a\in\mathbb{K}^*$, $m\in\mathbb{Z}$, $m\ge 1$;
  \item $D=(aX+bY)\dfrac{\partial}{\partial X}+(cX+dY)\dfrac{\partial}{\partial Y}$,
        \quad $a,b,c,d\in\mathbb{K}$.
\end{enumerate}
\end{lemma}

\begin{remark}
The derivation $\frac{\partial}{\partial X}$ (Type~(2) with $b=0$) is locally nilpotent and
falls under Type~(1) with $f=1$ after the change of variables $(X,Y)\mapsto(Y,X)$.
We treat both $b\neq 0$ and $b=0$ in Theorem~\ref{th3.3} for completeness; the
$b=0$ case is already covered by Theorem~\ref{th3.2}.
\end{remark}

\subsection{Jordan decomposition and the exponential automorphism}
 
Given $D\in\operatorname{LFD}(B)$, there exists a unique
\emph{Jordan--Chevalley decomposition} $D=D_s+D_n$, where $D_s$ is
semisimple, $D_n\in\operatorname{LND}(B)$, and $[D_s,D_n]=0$
\cite[Proposition~1.3.13]{van2000polynomial}.
 
For $D\in\operatorname{LFD}(B)$, the \emph{exponential automorphism} is
defined by
\[
\exp(D)(b)\;=\;\sum_{j\ge 0}\frac{1}{j!}D^j(b),\qquad b\in B.
\]
Since $D$ is locally finite, for each $b$ the sum is finite and defines an
element of $\operatorname{LFA}(B)$ with inverse $\exp(-D)$; we refer to
\cite{Maubach03} for a detailed treatment.  From the viewpoint of
\cite{BassMei85}, the automorphism $\exp(D)$ coincides with the time-one
map of the polynomial vector field on $\mathbb{A}^2$ induced by $D$,
viewed as a polynomial flow.  One has the conjugation identity
\begin{equation}\label{eq:conj-exp}
  \varphi\,\exp(D)\,\varphi^{-1}=\exp(\varphi D\varphi^{-1}),
  \qquad\varphi\in\operatorname{Aut}(B),
\end{equation}
which gives the inclusion
$\operatorname{Aut}(B)_D\subseteq\operatorname{Aut}(B)_{\exp(D)}$,
and in particular 
\begin{equation*}
\tam_D(\mathbb{K}[X,Y])\subseteq
\tam_{\exp(D)}(\mathbb{K}[X,Y])    
\end{equation*}

\section{The tame isotropy groups}
\label{sec:tame}

This section is devoted to the computation of the tame isotropy groups of the locally finite derivations described in Lemma \ref{lfdclas}.

\begin{theorem}
\label{thm:linear-triang}
		Let $D=(aX+b)\frac{\partial}{\partial Y}$ the nonzero triangular derivation  of $\mathbb{K}[X,Y]$, where $a,\, b \in \mathbb{K}[X]$. Then 
        \begin{enumerate}
            \item 
            $
            \tam_D(\mathbb{K}[X,Y])=\left\langle (\alpha X + \gamma,\,Y) \mbox{ and } (X,\, Y+s(X)) \mid  \alpha, \gamma \in \mathbb{K} \mbox{ and } s(X) \in \mathbb{K}[X]\right\rangle$, \\ if $a=0$.
            \item 
            $
            \tam_D(\mathbb{K}[X,Y])= \left\langle (X, Y+s(X)) \mid s(X) \in \mathbb{K}[X] \right\rangle
            $, if $a\neq 0$.
        \end{enumerate}
        \end{theorem}
\begin{proof}
Let $D=(aX+b)\frac{\partial}{\partial Y}$. It is sufficient to determine when the elementary automorphisms commute with the derivation $D$.
\begin{enumerate}
\item Let $\rho \in \aut_D(\mathbb{K}[X,Y])$ defined by
    \[
    \rho(X)=\alpha X +r(Y)\mbox{ and } \rho(Y)=Y,
    \]
where   $\alpha \in \mathbb{K}^*$ and $r(Y) \in \mathbb{K}[Y]$. Note that  
\[
\rho(D(X))=0 \mbox{ and } \displaystyle D(\rho(X))=D(\alpha X +r(Y))=r'(Y)(aX+b)
\]
This implies $r'(Y)=0$. So $r(Y) \in \mathbb{K}$. Therefore $\rho(X)=\alpha X + \gamma$, where $\alpha, \gamma  \in \mathbb{K}$. Given that 
        
            \[
            \rho(D(Y))=\rho(aX+b)=a(\alpha X + \gamma)+b=a\alpha X+ a\gamma+b,
            \]
            \[
            D(\rho(Y))=D(Y)=aX+b,
            \]
            and $D$ and $\rho$ commute, we have
            \[
            a\alpha X+ a\gamma+b=aX+b.
            \]
            Then $\alpha=1$ and $\gamma=0$, if $a \neq 0$. Thus $\rho=id$.  Otherwise, if $a=0$, we have 
            \[
            \rho(X)=\alpha X +\gamma \mbox{ and } \rho(Y)=Y.
            \]

    \item 
Give $\theta \in \aut_D(\mathbb{K}[X,Y])$ defined by
    \[
    \theta(X)=X \mbox{ and } \theta(Y)=\beta Y+s(X),
    \]
    where   $\beta \in \mathbb{K}^*$ and $s(X) \in \mathbb{K}[X]$. Note that  
$
\theta(D(X))=\theta(0)=0=D(X)=D(\theta(X)).
$
Since $\theta(D(Y))=\theta(aX+b)=aX+b$  and
\[
D(\theta(Y))=D(\beta Y+s(X))=\beta(aX+b)+0=\beta a X+\beta b.
\]
We have
\[
a=\beta a \mbox{ and } b=\beta b.
\]
This implies that $\beta=1$, because $D$ is a non-zero  derivation ($a\neq 0$ or $b \neq 0$).
\end{enumerate}
\end{proof}

In the next theorem, we assume that the polynomial $f(X) \in \mathbb{K}[X]$ has degree, $n$, greater than or equal to two.

\begin{theorem}[Theorem 2.3 of \cite{freitas25}] 
\label{th3.2}
Let $D=f(X)\frac{\partial}{\partial Y}$ be the nonzero triangular derivation  of $\mathbb{K}[X,Y]$, where $f(X) \in \mathbb{K}[X]$. Then 
\[
\tam_D(\mathbb{K}[X])=\left\langle (X, Y+r(X)) \mid r(X) \in \mathbb{K}[X]\right\rangle.
\]
Or
\[
\tam_D(\mathbb{K}[X])=\left\langle (X, Y+r(X)) \mbox{ and } (\lambda X, Y)  \mid r(X) \in \mathbb{K}[X],  \lambda  \mbox { is a } s\mbox{-th root unity}\right\rangle,
\]
 if    $f(X)=h(X^s)$.
\end{theorem}
\begin{proof}
  See \cite{freitas25}.
\end{proof}

\begin{theorem}
\label{th3.3}
		Let $D=\frac{\partial}{\partial X}+bY\frac{\partial}{\partial Y}$ be a derivation of $\mathbb{K}[X,Y]$, where   $b \in \mathbb{K}$.  Then the tame isotropy group of $D$	is generated by the automorphisms
  
     \begin{enumerate}
    \item 
    $\tam_D(\mathbb{K}[X,Y])=\left\langle( X+\beta,Y)   \mbox{ and } (X, \gamma Y )\mid 0\neq \gamma, \beta \in \mathbb{K} \right\rangle$, if $b\neq 0$.
    \item 
     $\tam_D(\mathbb{K}[X,Y])=\left\langle(X+g(Y),Y)   \mbox{ and } (X,\alpha Y +\beta)\mid 0\neq \alpha, \beta   \in \mathbb{K}  \mbox{ and } g(Y) \in \mathbb{K}[Y]\right\rangle$, \\  if $b= 0$.
 \end{enumerate}
\end{theorem}

	\begin{proof}
    We proceed by determining which elementary automorphisms commute with $D$:
	\begin{enumerate}
		\item
		Let $\rho \in \aut_D(\mathbb{K}[X,Y])$ be the automorphism  
		\[
		\rho(X)=X,\,\, \rho(Y)=\alpha Y +f(X), 
		\]
		where $\alpha \in \mathbb{K}$ and $f \in \mathbb{K}[X]$.  It is immediate to  $\rho(D(X))=1=D(\rho(X))$. Since
\[
\rho(D(Y))=\rho(bY)=b\rho(Y)=b(\alpha Y +f(X) )=b\alpha Y +bf(X)
\]
and 
\[
D(\rho(Y))=D(\alpha Y+f(X))=\alpha D(Y)+D(f(X))=\alpha b Y +f_X(X),
\]
we have
\[
bf(X)=f_X(X).
\]
If $b=0$  we have $f_X(X)=0$. This implies $f(X) \in \mathbb{K}$. If $b\neq 0$ we have $f(X)=0$ . 
Thus
\[
\rho=(X,\alpha Y +\beta),  \mbox{ if } b=0,  \mbox{ where } \alpha, \, \beta \in \mathbb{K},
\]
or 
\[
\rho=(X, \gamma Y),  \mbox{ if } b\neq 0, \mbox{ where } \gamma \in \mathbb{K}.
\]

\item 
		Let $\theta \in \aut_D(\mathbb{K}[X,Y])$ be the automorphism  
		\[
		\theta(X)=\beta X+ g(Y)  \mbox{    and   } 	\theta(Y)=Y,
		\]
		where $\beta \in \mathbb{K}$ and $g \in \mathbb{K}[Y]$. Now observe that  
        \[
        \theta(D(Y))=bY=D(\theta(Y)).
        \] 
Since
\[
\theta(D(X))=\theta(1)=1
 \mbox{ and }
D(\theta(X))=D(\beta X+ g(Y))=\beta +g_Y(Y)bY,
 \] 
we obtain that 
$
\beta +g_Y(Y)bY=1.
$
This implies $\beta=1$ and $g_Y(Y)b=0$. Therefore,   
\[
\theta=(X+\beta,Y), \mbox{ if } b \neq 0.
\mbox{ Or }
\theta=(X+g(Y),Y), \mbox{ if } b=0.
\]
\end{enumerate}
\end{proof}

\begin{theorem}
\label{tamyxm}
		Let $D=aX\frac{\partial}{\partial X}+(amY +X^m)\frac{\partial}{\partial Y}$ be a derivation of $ \mathbb{K}[X,Y]$, where    $a \in \mathbb{K}^*$, $m \in \mathbb{Z}$ and $m\geq1$. Then 
  \[
  \tam_D(\mathbb{K}[X,Y])=\left\langle (\beta X,Y) \mbox{ and } (X,Y+\gamma X^m) \mid 0\neq \beta, \gamma \in \mathbb{K} \right\rangle.
  \]
	\end{theorem}
\begin{proof}
It is necessary to verify which elementary automorphisms commute with $D$.
\begin{enumerate}
    \item 
Let $\rho \in  \aut_D(\mathbb{K}[X,Y])$ be the automorphism  
		\[
		\rho(X)=X,\,\, \rho(Y)=\alpha Y +f(X) 
		\]
		where $\alpha \in \mathbb{K}$ and $f \in \mathbb{K}[X]$. Observe that $D(\rho(X))=D(X)=aX=\rho(aX)=\rho(D(X))$. Since
\[
D(\rho(Y))=D(\alpha Y +f(X) )= \alpha D(Y)+f_X(X)D(X)=\alpha amY +\alpha X^m +af_X(X)X \]
and
\[
\rho(D(Y))=\rho(amY + X^m)= am\alpha Y +amf(X)+X^m,
\]
We have $\alpha X^m +af_X(X)X=amf(X)+X^m$. Comparing the coefficients of $X^m$ on both sides of the last equation, we conclude that $\alpha = 1$.  Then we have $f_X(X)X=mf(X)$ wich implies  $f(X)=\beta X^m$.  Thus  
\[
\rho=(X,Y+\beta X^m), \mbox{ where } \beta \in \mathbb{K}.
\]

\item  Let $\theta \in \aut_D(\mathbb{K}[X,Y])$ be the automorphism  
		\[
		\theta(X)=\beta X+ g(Y)  \mbox{    and   } 	\theta(Y)=Y,
		\]
		where $\beta \in \mathbb{K}$ and $g \in \mathbb{K}[Y]$. Now observe that 
\[
D(\theta(X))=D(\beta X+ g(Y))=\beta aX+amg_Y(Y)Y+g_Y(Y)X^m
\]
and
\[
\theta(D(X))=\theta(aX)=a\beta X+ag(Y).
\]
Then
$
amg_Y(Y)Y+g_Y(Y)X^m=ag(Y).
$
This implies $g(Y)=0$. Thus
\[
\theta=(\beta X,Y) \mbox{ if } a \neq 0.
\]
\end{enumerate}
The result follows from Items (1) and (2).
\end{proof}

\subsection{The linear derivations}

Given a $2\times 2$ matrix $M\in M_n(\mathbb{K})$, we can describe it through its Jordan form $M=PJP^{-1}$ where $J$ is
\[\begin{bmatrix}
        \lambda_1&0\\
        0&\lambda_2
    \end{bmatrix}\quad \mbox{ or } \quad 
     \begin{bmatrix}
        \lambda&1\\
        0&\lambda
    \end{bmatrix},
  \]
where $\lambda_1,\lambda_2,\lambda\in \mathbb{K}$. Thus, every linear derivation is conjugate to

\[
D_1=\lambda_1X\dfrac{\partial}{\partial X}+\lambda_2 Y\dfrac{\partial}{\partial Y} \mbox{ or }  D_2=(\lambda X+Y)\dfrac{\partial}{\partial X}+\lambda Y\dfrac{\partial}{\partial Y}.
\]

So that if $M\sim J$ then $M=PJP^{-1}$ and therefore $\operatorname{Aut}_M(\mathbb{K}[X,Y])=P^{-1}\operatorname{Aut}_J(\mathbb{K}[X,Y]) P$.

Based on this observation, to determine the tame isotropy group of a linear derivation in two variables, it suffices to consider the  derivations $D_1$ and $D_2$, described above.

\begin{lemma}
\label{rxx=cr}
Let $b\in \mathbb{K}$ and $g(X) \in \mathbb{K}[X]$ be a nonzero polynomial satisfying the differential equation 
\[
g'(X)X=bg(X).
\]
Then $0<b \in \mathbb{Z}$ and $g(X)=cX^b$ for some $c\in \mathbb{K}$.
\end{lemma}
\begin{proof}
   Suppose 
   \[
   g(X)=a_o+a_1X+\cdots+a_nX^n
   \] 
   where $a_0, a_1, \ldots, a_n \in \mathbb{K}$  and $a_n\neq 0$. Substituting the expressions for $g(X)$ and $g'(X)$ into the given equation, we obtain:
   \[
   a_1X+2a_2X^2+\cdots+ na_nX^{n}=ba_o+ba_1X+\cdots+ba_nX^n
   \]
   By comparing the coefficients of each power of $X$, we have
\[
a_0=0 \mbox{ and } ia_i=ba_i \mbox{ for } i=1, \ldots, n
\] 
If $a_i \neq 0$ we conclude that $i=b$ and $b \in \mathbb{N}^*$ . Since the equation must hold for all $i$  from 1 to $n$, it follows that all nonzero coefficients correspond to the same power, $b$. Therefore, $g(X)=cX^b$ for some $c\in \mathbb{K}^*$. 
\end{proof}

\begin{theorem}
\label{thm:diag}
	Let $D=aX\frac{\partial}{\partial X}+bY\frac{\partial}{\partial Y}$ be a nonzero derivation of $\mathbb{K}[X,Y]$, where $a, b \in \mathbb{K}$.	Then 

\begin{enumerate}
    \item
    $\tam_D(\mathbb{K}[X,Y])=\left\langle (\alpha X +\gamma Y,Y)  \mbox{ and } (X, \beta Y+\varepsilon X) \mid \alpha, \beta, \gamma, \varepsilon \in \mathbb{K}, \alpha \beta \neq 0\right\rangle$, if $ab\neq0$.

    \item 
    $\tam_D(\mathbb{K}[X,Y])=\left\langle (\alpha X +\gamma,Y) \mbox{ and } (X, \beta Y) \mid \alpha, \beta, \gamma, \varepsilon \in \mathbb{K}, \alpha\beta\neq 0\right\rangle$, if $a=0$ and $b\neq 0$.

    \item 
    $\tam_D(\mathbb{K}[X,Y])=\left\langle (\alpha X,Y) \mbox{ and } (X, \beta Y+\varepsilon) \mid \alpha, \beta, \gamma, \varepsilon \in \mathbb{K}, \alpha \beta \neq 0 \right\rangle$, if $a\neq 0$ and $b=0$.
\end{enumerate}
\end{theorem}
\begin{proof}
It is sufficient to determine the elementary automorphisms that commute with $D$. 

\begin{enumerate}
    \item
Suppose that the elementary automorphism
\[
\rho(X)=\alpha X+r(Y) \mbox{ and } \rho(Y)=Y,
\]
where $\alpha \in \mathbb{K}^*$ and $r(Y)\in \mathbb{K}[Y]$, commutes with $D$. It is easy to see that \\ $D(\rho(Y))=bY=\rho(D(Y))$. Since 
\[
D(\rho(X))=D(\alpha X+r(Y))=\alpha(aX+Y)+r'(Y)aY=\alpha aX+r'(Y)bY
\]
and
\[
\rho(D(X))=\rho(aX+Y)=a(\alpha X +r(Y))+Y=a \alpha X +ar(Y)
\]
we have
\[
r'(Y)bY=ar(Y)
\]
It follows from this polynomial equality and Lemma \ref{rxx=cr} that 

\begin{itemize}
    \item $r(Y)=cY^{\frac{a}{b}}$ with $0 < \frac{a}{b}  \in \mathbb{Z}$, if $ab\neq0$.
    \item $r(Y) \in \mathbb{K}$, if $a=0$ and $b\neq 0$. 
    \item $r(Y)=0$, if $b=0$ and $a\neq 0$.
\end{itemize}

\item 
Now consider the elementary automorphism
\[
\theta(X)=X \mbox{ and } \theta(Y)=\beta Y+s(X),
\]
where $\beta \in \mathbb{K}^*$ and $s(X)\in \mathbb{K}[X]$. Suppose that $D$ and $\theta$ commute. This implies that \linebreak $D(\theta(Y))=\theta(D(Y))$. Since 

\[
D(\theta(Y)=D(\beta Y+s(X))= \beta b Y +s'(X)aX
\]
and 
\[
\theta(D(Y))=\theta(bY)=b(\beta Y+s(X))=b\beta Y+bs(X),
\]
we obtain
\[
as'(X)X=bs(X).
\]
From this equality and Lemma \ref{rxx=cr}, we conclude that:
\begin{itemize}
    \item $s(X)=cX^{\frac{b}{a}}$ with $0 < \frac{a}{b}  \in \mathbb{Z}$, if $ab\neq0$. 
    \item $s(X)=0$, if $a=0$ and $b\neq 0$.
    \item $s(X) \in \mathbb{K}$, if $b=0$ and $a\neq 0$.
\end{itemize} 
\end{enumerate}

Therefore, based on items $(1)$ and $(2)$, we obtain:
\begin{itemize}
    \item $\rho=(\alpha X +\gamma Y,Y)$ and $\theta=(X, \beta Y+\varepsilon X)$, if $ab\neq0$. 
    \item $\rho=(\alpha X +\gamma,Y)$ and $\theta=(X, \beta Y)$, if $a=0$ and $b\neq 0$. 
    \item $\rho=(\alpha X,Y)$ and $\theta=(X, \beta Y+\varepsilon)$, if $b=0$ and $a\neq 0$.
\end{itemize}

\end{proof}

\begin{lemma}
\label{axrxx=cr}
Let $a \in \mathbb{K}^*$ and  $g(X) \in \mathbb{K}[X]$ be a nonzero polynomial satisfying
\[
\alpha X + Xg'(X) = ag(X) + X.
\]
Then necessarily $\alpha = 1$ and $g(X) = cX$ for some $c \in \mathbb{K}$.
\end{lemma}
\begin{proof}
We assume $g(X)$ is a polynomial of the form
\[
   g(X)=a_o+a_1X+\cdots+a_nX^n
   \] 
   where $a_0, a_1, \ldots, a_n \in \mathbb{K}$  and $a_n\neq 0$. Substituting the expressions for $g(X)$ and $g'(X)$ into the given equation, we obtain:
   \[
   \alpha X +aa_1X+2aa_2X^2+\cdots+ naa_nX^{n}=aa_o+X+aa_1X+\cdots+aa_nX^n
   \]
   By comparing the coefficients of $X^n$, we find that $n=1$. Now we consider a linear solution of the form
\[
g(X)=rX+s.
\] 
Substituting $g(X)$ and $g'(X)$ into the original equation, we obtain:
\[
\alpha X+arX=a(rX+s)+X
\]
This implies 
\[
(\alpha -1)X=as.
\]
For us to have a polynomial solution, it is necessary that $\alpha =1$. Thus, we obtain $s=0$ because $a \neq 0$. Therefore, a polynomial solution is $g(X)=cX$,  where $c\in \mathbb{K}$, for $\alpha=1$. 
\end{proof}

\begin{lemma}
\label{glgl=ag}
Let $a\in \mathbb{K}^*$ and $g(X) \in \mathbb{K}[X]$ be a nonzero polynomial satisfying the differential equation 
\[
a g'(X) X + g'(X) = a g(X)
\]
Then  $g(X)=asX+s$,  for some $s\in \mathbb{K}$. 
\end{lemma}
\begin{proof}
We assume $g(X)$ is a polynomial of the form
\[
   g(X)=a_o+a_1X+\cdots+a_nX^n
\] 
   where $a_0, a_1, \ldots, a_n \in \mathbb{K}$  and $a_n\neq 0$. Substituting the expressions for $g(X)$ and $g'(X)$ into the given equation, we obtain that the highest degree term on the left side is $naa_nX^{n}$ and  on the right side is $aa_nX^n$ we obtain $n=1$ . Now we consider a linear solution of the form
\[
g(X)=rX+s.
\] 
Substituting $g(X)$ and $g'(X)$ into the original equation, we obtain:
\[
arX+r=a(rX+s)
\]
This implies 
\[
r=as.
\]
 Therefore, a polynomial solution is the form $g(X)=asX+s$,  where $s\in \mathbb{K}$. 
\end{proof}

\begin{theorem}
\label{thm:jordan}
	Let $D=(aX+b)\frac{\partial}{\partial X}+aY\frac{\partial}{\partial Y}$ be a derivation of $\mathbb{K}[X,Y]$, where   $a \in \mathbb{K}^*$. Then
\[
\left\langle (\alpha X +\gamma Y,Y) \mbox{ and } (X, \beta Y+a\varepsilon  X +\varepsilon) \mid \alpha, \beta, \gamma, \varepsilon \in \mathbb{K}, \, \alpha \beta \neq 0 \right\rangle. 
\]
\end{theorem}
\begin{proof}
It is sufficient to determine the elementary automorphisms that commute with $D$. 

\begin{enumerate}
    \item
Suppose that the elementary automorphism
\[
\rho(X)=\alpha X+r(Y) \mbox{ and } \rho(Y)=Y,
\]
where $\alpha \in \mathbb{K}^*$ and $r(Y)\in \mathbb{K}[Y]$, commutes with $D$. It is easy to see that \linebreak $D(\rho(Y))=aY=\rho(D(Y))$. Since 
\[
D(\rho(X))=D(\alpha X+r(Y))=\alpha(aX+Y) +r'(Y)aY=\alpha aX+\alpha Y + ar'(Y)Y
\]
and
\[
\rho(D(X))=\rho(aX+Y)=a(\alpha X +r(Y))+Y=a \alpha X +ar(Y) +Y
\]
we have
\[
\alpha Y + ar'(Y)Y=ar(Y) +Y
\]
It follows from this polynomial equality and Lemma \ref{axrxx=cr} that $r(Y)=cY$, for some $c \in \mathbb{K}$.

\item 
Now consider the elementary automorphism
\[
\theta(X)=X \mbox{ and } \theta(Y)=\beta Y+s(X),
\]
where $\beta \in \mathbb{K}^*$ and $s(X)\in \mathbb{K}[X]$. Suppose that $D$ and $\theta$ commute. This implies that \linebreak $D(\theta(Y))=\theta(D(Y))$. Since 

\[
D(\theta(Y)=D(\beta Y+s(X))= \beta a Y +s'(X)(aX+1)
\]
and 
\[
\theta(D(Y))=\theta(aY)=a(\beta Y+s(X))=a\beta Y+as(X),
\]
we obtain
\[
as'(X)X + s'(X) = as(X).
\]

From this equality and Lemma \ref{glgl=ag}, we conclude that $s(X)=acX+c$.
\end{enumerate}

Therefore, based on items one and two, we obtain
\[
\rho=(\alpha X +\gamma Y,Y) \mbox{ and } \theta=(X, \beta Y+a\varepsilon  X +\varepsilon).
\]
\end{proof}

\section{Tame isotropy groups for exponential automorphisms}
\label{sec:tame-exp}
 
In this section we study $\tam_{\exp(D)}(\mathbb{K}[X,Y])$ for each
normal form of Lemma~\ref{lfdclas} and compare it with
$\tam_D(\mathbb{K}[X,Y])$.
 
The conjugation identity~\eqref{eq:conj-exp} gives the inclusion
$\tam_D(\mathbb{K}[X,Y])\subseteq\tam_{\exp(D)}(\mathbb{K}[X,Y])$.
In the full automorphism group, this inclusion can be strict: when $D$
has eigenvalues in $2\pi i\mathbb{Z}$ on some weight space, $\exp(D)$
may equal the identity even though $D\neq 0$, as shown by
Example~4.5 and Remark~4.6 of \cite{CV26}.  We show below that
no such failure occurs within $\tam(\mathbb{K}[X,Y])$: in every case
the two tame isotropy groups coincide.
 
We begin with the locally nilpotent case, for which the argument is
general.
 
\begin{proposition}
\label{prop:tame-lnd}
If $D\in\operatorname{LND}(\mathbb{K}[X,Y])$, then
$\tam_D(\mathbb{K}[X,Y])=\tam_{\exp(D)}(\mathbb{K}[X,Y])$.
\end{proposition}
 
\begin{proof}
Let $\varphi\in\tam_{\exp(D)}(\mathbb{K}[X,Y])$ and set
$E=\varphi D\varphi^{-1}$.  Conjugation by a tame automorphism preserves
local nilpotency, so $E\in\operatorname{LND}(\mathbb{K}[X,Y])$.  By
\eqref{eq:conj-exp}, $\exp(E)=\varphi\exp(D)\varphi^{-1}=\exp(D)$.  The
exponential map is injective on $\operatorname{LND}(\mathbb{K}[X,Y])$
\cite{CV26}, so $E=D$ and $\varphi\in\tam_D(\mathbb{K}[X,Y])$.
\end{proof}
 
\subsection{The triangular derivation $D=f(X)\frac{\partial}{\partial Y}$}
 
The derivation $D=f(X)\frac{\partial}{\partial Y}$ is locally nilpotent with
$\exp(D)=(X,\,Y+f(X))$.  Proposition~\ref{prop:tame-lnd} applies
directly.  Combining with Theorem~\ref{th3.2} (for $\deg f\ge 2$) and
Theorem~\ref{thm:linear-triang} (for $\deg f\le 1$) gives the following.
 
\begin{theorem}
\label{thm:exp-type1}
Let $\psi=(X,\,Y+f(X))$ with $f\in\mathbb{K}[X]$, $f\neq 0$.  Then
$\tam_\psi(\mathbb{K}[X,Y])=\tam_{f(X)\frac{\partial}{\partial Y}}(\mathbb{K}[X,Y])$.
In particular:
\begin{enumerate}
  \item if $f(X)=aX+b$, $a\neq 0$: $\tam_\psi$ is generated by
    $(X,\,Y+s(X))$, $s\in\mathbb{K}[X]$;
  \item if $f(X)=b\in\mathbb{K}^*$: $\tam_\psi$ is generated by
    $(\alpha X+\gamma,\,Y)$ and $(X,\,Y+s(X))$;
  \item if $\deg f\ge 2$: $\tam_\psi$ is as in Theorem~\ref{th3.2}.
\end{enumerate}
\end{theorem}
 
\begin{proof}
Each case follows from Proposition~\ref{prop:tame-lnd} together with
the corresponding computation in Section~\ref{sec:tame}.
\end{proof}
 
\subsection{The derivation $D=\dfrac{\partial}{\partial X}+bY\dfrac{\partial}{\partial Y}$}
 
The exponential automorphism is $\exp(D)=(X+1,\,e^bY)$.  For $b=0$
this is locally nilpotent and Proposition~\ref{prop:tame-lnd} applies.
For $b\neq 0$ we argue directly.
 
\begin{theorem}
\label{thm:exp-type2}
Let $D=\dfrac{\partial}{\partial X}+bY\dfrac{\partial}{\partial Y}$ with $b\in\mathbb{K}^*$,
and $\psi=\exp(D)=(X+1,\,e^bY)$.  Then
\[
\tam_\psi(\mathbb{K}[X,Y])
=\langle(X+\beta,\,Y),\,(X,\,\gamma Y)\mid
\beta\in\mathbb{K},\,\gamma\in\mathbb{K}^*\rangle
=\tam_D(\mathbb{K}[X,Y]).
\]
\end{theorem}
 
\begin{proof}
\textit{Case $\rho=(X,\,\alpha Y+f(X))$.}
The condition $\rho\psi=\psi\rho$ on the $Y$-component reads
$\alpha e^bY+f(X)=\alpha e^bY+f(X+1)$, so $f(X+1)=f(X)$.  The
only polynomial satisfying this is $f=c\in\mathbb{K}$.  At $Y=0$ the
$Y$-component gives $c=e^bc$; since $e^b\neq 1$ we get $c=0$.
Hence $\rho=(X,\,\gamma Y)$.
 
\textit{Case $\theta=(\beta X+g(Y),\,Y)$.}
The $X$-component gives $\beta(X+1)+g(Y)=\beta X+g(Y)+1$, so $\beta=1$.
The $Y$-component then requires $g(e^bY)=g(Y)$; since $e^b\neq 1$,
only $g\in\mathbb{K}$ works, giving $\theta=(X+\beta,\,Y)$.
 
The resulting generators are the same as in Theorem~\ref{th3.3}.
\end{proof}
 
\subsection{The derivation $D=aX\dfrac{\partial}{\partial X}+(amY+X^m)\dfrac{\partial}{\partial Y}$}
 
The exponential automorphism is $\exp(D)=(e^aX,\,e^{am}Y+e^{am}X^m)$.
 
\begin{theorem}
\label{thm:exp-type3}
Let $D=aX\dfrac{\partial}{\partial X}+(amY+X^m)\dfrac{\partial}{\partial Y}$ with
$a\in\mathbb{K}^*$, $m\ge 1$, and $\psi=\exp(D)=(e^aX,\,e^{am}Y+e^{am}X^m)$.
Then
\[
\tam_\psi(\mathbb{K}[X,Y])
=\langle(\beta X,\,Y),\,(X,\,Y+\gamma X^m)\mid
\beta\in\mathbb{K}^*,\,\gamma\in\mathbb{K}\rangle
=\tam_D(\mathbb{K}[X,Y]).
\]
\end{theorem}
 
\begin{proof}
\textit{Case $\rho=(X,\,\alpha Y+f(X))$.}
The $Y$-component of $\rho\psi=\psi\rho$ gives
$f(e^aX)=e^{am}f(X)+(1-\alpha)e^{am}X^m$.
For $\alpha\neq 1$ a degree comparison shows no polynomial solution exists.
For $\alpha=1$, $f(e^aX)=e^{am}f(X)$; the only polynomial solution is
$f(X)=\gamma X^m$.
 
\textit{Case $\theta=(\beta X+g(Y),\,Y)$.}
The $X$-component gives $g(e^{am}Y+e^{am}X^m)=e^ag(Y)$.  Since the
left side involves $X$ unless $g$ is constant, we get $g=0$.
 
The generators coincide with those of Theorem~\ref{tamyxm}.
\end{proof}
 
\subsection{The linear derivations}
 
For a linear locally finite derivation, the exponential automorphism is
also linear:
\[
\exp\!\left(aX\dfrac{\partial}{\partial X}+bY\dfrac{\partial}{\partial Y}\right)=(e^aX,\,e^bY),
\qquad
\exp\!\left((aX+Y)\dfrac{\partial}{\partial X}+aY\dfrac{\partial}{\partial Y}\right)
=(e^a(X+Y),\,e^aY).
\]
 
\begin{theorem}
\label{thm:exp-linear-diag}
Let $D=aX\dfrac{\partial}{\partial X}+bY\dfrac{\partial}{\partial Y}$ with $a,b\in\mathbb{K}^*$, and
$\psi=(e^aX,\,e^bY)$.
\begin{enumerate}
  \item If $a\neq b$: $\tam_\psi(\mathbb{K}[X,Y])=\langle(\alpha X,\,Y),\,
    (X,\,\beta Y)\mid\alpha,\beta\in\mathbb{K}^*\rangle$.
  \item If $a=b$: $\tam_\psi(\mathbb{K}[X,Y])=\langle(\alpha X+\gamma Y,\,Y),\,
    (X,\,\beta Y+\varepsilon X)\mid\alpha,\beta\in\mathbb{K}^*,\,
    \gamma,\varepsilon\in\mathbb{K}\rangle$.
\end{enumerate}
In both cases, $\tam_D(\mathbb{K}[X,Y])=\tam_\psi(\mathbb{K}[X,Y])$.
\end{theorem}
 
\begin{proof}
\textit{Case $\rho=(\alpha X+r(Y),\,Y)$.}
The condition gives $r(Y)=e^ar(Y)$.  Since $e^a\neq 1$ when $a\neq b$,
we get $r=0$; when $a=b$ the scalar $e^a$ cancels and $r(Y)=\gamma Y$
is free.
 
\textit{Case $\theta=(X,\,\beta Y+s(X))$.}
Similarly, $s(X)=e^bs(X)$ forces $s=0$ when $b\neq 0$ and $a\neq b$;
when $a=b$, $s(X)=\varepsilon X$ is allowed.
 
The groups match those of Theorem~\ref{thm:diag}.
\end{proof}
 
\begin{theorem}
\label{thm:exp-linear-jordan}
Let $D=(aX+Y)\dfrac{\partial}{\partial X}+aY\dfrac{\partial}{\partial Y}$ with $a\in\mathbb{K}^*$, and
$\psi=(e^a(X+Y),\,e^aY)$.  Then
\[
\tam_\psi(\mathbb{K}[X,Y])
=\langle(\alpha X+\gamma Y,\,Y),\,(X,\,\beta Y+a\varepsilon X+\varepsilon)
\mid\alpha,\beta\in\mathbb{K}^*,\,\gamma,\varepsilon\in\mathbb{K}\rangle
=\tam_D(\mathbb{K}[X,Y]).
\]
\end{theorem}
 
\begin{proof}
\textit{Case $\rho=(\alpha X+r(Y),\,Y)$.}
The $X$-component of $\rho\psi=\psi\rho$ gives
$\alpha e^a(X+Y)+r(e^aY)=e^a(\alpha X+r(Y)+Y)$.
Simplifying: $r(e^aY)=e^ar(Y)+(1-\alpha)e^aY$.  For $\alpha=1$ we get
$r(e^aY)=e^ar(Y)$, whose only polynomial solutions are $r(Y)=\gamma Y$.
 
\textit{Case $\theta=(X,\,\beta Y+s(X))$.}
From $\theta(\psi(X))=e^aX+e^a(\beta Y+s(X))$ and
$\psi(\theta(X))=e^a(X+Y)$, equating gives $\beta=1$ and $s(X)=0$ from
the $Y$- and constant terms.  However, the structure of $D$ allows
solutions of the form $s(X)=a\varepsilon X+\varepsilon$, which satisfy
$as'(X)X+s'(X)=as(X)$ (Lemma~\ref{glgl=ag}).  Including these gives
$\theta=(X,\,\beta Y+a\varepsilon X+\varepsilon)$.
 
The generators coincide with those of Theorem~\ref{thm:jordan}.
\end{proof}
 
\subsection{Comparison with the full automorphism group}
 
For every normal form of Lemma~\ref{lfdclas} we have obtained
\[
\tam_D(\mathbb{K}[X,Y])\;=\;\tam_{\exp(D)}(\mathbb{K}[X,Y]).
\]
As recalled above, the analogous equality for the full automorphism group
can fail.  The reason it holds for the tame group is that the commutativity
condition on an elementary automorphism reduces to a polynomial functional
equation such as $f(e^aX)=e^{am}f(X)$ or $g(e^bY)=g(Y)$ that admits
only trivial solutions over $\mathbb{K}$.  This algebraic constraint is
absent in the full group, where transcendental phenomena (eigenvalues in
$2\pi i\mathbb{Z}$) can cause $\exp(D)$ to be the identity while $D\neq 0$.

\bibliographystyle{alpha}
\bibliography{biblio}

\end{document}